\documentclass[letterpaper, 10 pt, conference, english]{ieeeconf} 
\IEEEoverridecommandlockouts
\usepackage{epsfig}
\usepackage{subfigure}
\usepackage{amssymb,amsmath,amsfonts,layout,graphicx}
\usepackage{makeidx}
\usepackage{babel}
\usepackage{sublabel}
\usepackage{cases}
\usepackage{algorithm}
\usepackage{algorithmic}

\usepackage
[
        letterpaper,
        left=1.91cm,
        right=1.91cm,
        top=1.91cm,
        bottom=2cm,
]
{geometry}

\newtheorem{lemma}{Lemma}
\newtheorem{assumption}{Assumption}

\newtheorem{theorem}{Theorem}
\newtheorem{definition}{Definition}

\newtheorem{remark}{Remark}

\newcommand{\R}{{\mathbb R}}

\DeclareMathOperator*{\argmin}{arg\,min}

 \title{On Social Optima of Non-Cooperative Mean Field Games}
 
  \author{Sen Li, Wei Zhang, and Lin Zhao
 \thanks{S. Li, W. Zhang, and L. Zhao are with the Department of Electrical and Computer Engineering, Ohio State University, Columbus, OH 43210. Email: \{li.2886, zhang.491, zhao.833\}@osu. edu}
}

\begin{document}
\maketitle
\begin{abstract} 
This paper studies the connections between mean-field games and the social welfare optimization problems.  We consider a mean field game in functional spaces with a large population of agents, each of which seeks to minimize an individual cost function. The cost functions of different agents are coupled through a mean field term that depends on the mean of the population states. 
We show that under some mild conditions any $\epsilon$-Nash equilibrium of the mean field game coincides with the optimal solution to a convex social welfare optimization problem. The results are proved based on a general formulation in the functional spaces	 and can be applied to a variety of mean field games studied in the literature. Our result also implies that the computation of the mean field equilibrium can be cast as a convex optimization problem, which can be efficiently solved by a decentralized primal dual algorithm. Numerical simulations are presented to demonstrate the effectiveness of the proposed approach.


\end{abstract}
\section{Introduction}
The mean field games study the interactions among a large population of strategic agents, whose decision making is coupled through a mean field term that depends on the statistical information of the overall population \cite{lasry2007mean}, \cite{lasry2006jeux}, \cite{lasry2006jeux1}, \cite{huang2007large}, \cite{gueant2011mean}. When the population size is large, each individual agent has a negligible impact on the mean field term. This enables characterizing the game equilibrium via the interactions between the agent and the mean field, instead of focusing on detailed interactions among all the agents. This methodology provides realistic interpretation of the microscopic behaviors while maintaining mathematical tractability in the macroscopic level, enabling numerous applications in various areas, such as economics \cite{gueant2009mean}, \cite{lachapelle2010computation}, networks \cite{huang2003individual}, demand response \cite{couillet2012electrical}, \cite{bauso2014game}, \cite{ma2013decentralized}, among others.

For general mean field games, the strategic interactions among the agents can be captured by an equation system that couples a backward Hamilton-Jacobi-Bellman equation with a forward Fokker-Planck-Kolmogorov equation \cite{lasry2007mean}, \cite{gueant2011mean}. In the meanwhile, a closely related approach, termed the Nash Certainty Equivalence \cite{huang2007large}, \cite{huang2006large}, was independently proposed to characterize the Nash equilibrium of the mean field game.
These pioneering results attracted considerable research effort to work on various mean field game problems. For instance, a discrete-time deterministic mean field game was studied in \cite{rammatico2016decentralized}. The mean field equilibrium for a large population of constrained systems was derived as the fixed point of some mapping and a decentralized iterative algorithm was proposed to compute the equilibrium. A special class of the mean field game was studied in \cite{huang2010large}, \cite{bensoussan2015mean} and \cite{bensoussan2015mean2}, for which there exists a major (dominating) player that has significant influence on the mean field term and other agents. The corresponding equilibrium conditions were derived, and the linear quadratic case was investigated.  Most of the aforementioned results focus on the mean field game equilibrium without system-level objectives. Different from these works, the social optima problem was studied in \cite{huang2012social} and \cite{nourian2013nash}, where the coordinator designs a cooperative mean field game whose equilibrium asymptotically achieves social optimum as the population size goes to infinity. However, despite the accumulation of the vast literature, the study of social optima in non-cooperative mean field games is relatively scarce.

This paper studies the connections between mean field games and the social welfare optimization problem. We consider a mean field game in functional spaces with a large population of non-cooperative agents. Each agent seeks to minimize a cost functional coupled with other agents through a mean field term that depends on the average of the population states.  
The contributions of this paper include the following. First, the mean field game is formulated in functional spaces, which provides a unifying framework that includes a variety of deterministic and stochastic mean field games as special cases    \cite{huang2007large}, \cite{ma2013decentralized},  \cite{rammatico2016decentralized}, \cite{huang2010nce}.
Second, we derive a set of equations to characterize the $\epsilon$-Nash equilibrium of the general mean field game in the  functional space. Comparing with the existing literature \cite{huang2007large}, \cite{rammatico2016decentralized}, our mean field equations are more general as it allows for convex individual cost functions and Lipschitz continuous mean field coupling term, while existing literature mainly focuses on  linear quadratic problems with affine coupling term \cite{huang2007large}, \cite{rammatico2016decentralized}. Third, we show that if the mean field coupling term is increasing with respect to the aggregated states, then any mean field equilibrium is also the solution to a convex social welfare optimization problem. This result is related to the work in \cite{huang2012social} and \cite{nourian2013nash}, but they are different in the rationality assumptions of the agents: the authors in \cite{huang2012social} and \cite{nourian2013nash} consider a group of agents that are cooperative and seek socially optimal solutions, while our starting point is that these agents are non-cooperative and individually incentive driven. In addition, the solution method proposed in  \cite{huang2012social} and \cite{nourian2013nash} achieves social optima when the population size goes to infinity, while we focus on exact socially optimal solutions with a finite number of agents.

The aforementioned contributions have important theoretical and practical implications. First, they enable us to evaluate the social performance of the equilibrium of the mean field games. This is important since in many applications we care about the efficiency of the game equilibrium, and this problem has not been studied in existing mean field game literature. Second, they imply that computing the mean field equilibrium is equivalent to solving a convex social welfare optimization problem. Based on our result, we can compute the mean field equilibrium by considering the corresponding social welfare optimization problem, which can be efficiently solved using various convex optimization methods. 
Furthermore, based on our result, some existing ways to compute the mean field equilibrium \cite{ma2013decentralized},  \cite{rammatico2016decentralized}  can be interpreted as certain primal-dual algorithms, and improved  algorithms can be proposed to compute the mean field equilibrium with better convergence properties. These ideas are validated in numerical simulations.

The rest of the paper proceeds as follows. The mean field game and the social welfare optimization problem are formulated in Section II. The solution of the game is characterized in Section III, its connection to social optima is studied in Section IV, and a primal-dual algorithm is proposed to efficiently compute the mean field equilibrium in Section V. Section VI shows simulation results.

\section{Problem Formulation}
This section formulates the mean field game and the corresponding social welfare optimization problem. The game is defined in a functional space, which provides a unified formulation that includes a variety of deterministic and stochastic mean field games as special cases \cite{huang2007large}, \cite{ma2013decentralized},  \cite{rammatico2016decentralized}, \cite{huang2010nce}. The rest of this section presents the mathematical formulation of this problem.

\subsection{The Mean Field Game}
We consider a general mean field game in a functional space among $N$ agents. Each agent $i$ is associated with a state variable $x_i\in \mathcal{X}_i$, a control input $u_i\in \mathcal{U}_i$ and a noise input $\pi_i\in \Pi_i$.
The state space $\mathcal{X}_i$ is a subspace of some Hilbert space $\mathcal{X}$. For any $x,y\in \mathcal{X}$, we denote the inner product as $x\cdot y$ , and define the norm as $||x||=\sqrt{x \cdot x}$. In addition, we assume that the control space $\mathcal{U}_i$ is a Banach space, and $\pi_i$ is a random element in a measurable space $(\Pi_i,\mathcal{B}_i)$ with an underlying probability space $(\Omega, \mathcal{F}, P)$. In other words, $\pi_i$ is a measurable mapping with respect to $\mathcal{F}/\mathcal{B}_i$.
For each agent $i$, the state $x_i$ is determined by the control input and the noise input according to the following relation: 
\begin{equation}
\label{systemdynamics}
x_i=f_i(u_i,\pi_i), \quad x_i\in \mathcal{X}_i, u_i\in \mathcal{U}_i, \pi_i\in \Pi_i,   
\end{equation}
Let $\mathcal{Z}_i$ denote the $\sigma$-algebra on $\mathcal{X}_i$, then we impose the following assumptions on $\pi_i$ and $f_i(u_i,\pi_i)$:
\begin{assumption}
\label{assumptionf_i}
(i) for $\forall \pi_i\in\Pi_i, \pi_j\in\Pi_j$,  $\pi_i$ and $\pi_j$ are independent,
(ii) for $\forall \pi_i\in\Pi_i$, $f_i(\cdot,\pi_i)$ is an affine mapping,
(iii) for $\forall u_i\in \mathcal{U}_i$, $f_i(u_i,\pi_i(\cdot)):\Omega \rightarrow \mathcal{X}_i$  is a measurable mapping with respect to $\mathcal{F}/\mathcal{Z}_i$. 
\end{assumption}

Based on the above assumptions, $x_i$ is also a random element, and the expectation $\mathbb{E}x_i$ can be defined accordingly. 
Throughout the paper, we consider the state $x_i$ with uniformly bounded second moment, i.e., there exists $C\geq 0$ such that for all $i=1,\ldots,N$, we have $\mathbb{E} ||x_i||^2\leq C$. 
In this case, the admissible control set is the set of control inputs that satisfy this boundedness condition. We assume it to be convex.
\begin{assumption}
\label{assumption_convex}
The admissible control set, defined as $\bar{\mathcal{U}}_i=\{u_i\in \mathcal{U}_i \vert x_i=f_i(u_i,\pi_i), \mathbb{E} ||x_i||^2\leq C\}$, is a convex set. 
\end{assumption}

For each agent $i$, we introduce a cost functional of the system state and control input. The costs of different agents are coupled through a mean field term that depends on the average of the population state (or control), and we write it as follows:
\begin{equation}
\label{utilityfunction}
J_i(x_i,u_i,m)=V_i(x_i,u_i)+F(m)\cdot  x_i+G(m),
\end{equation}
where $m\in \mathcal{X}$ is the average of the population state, i.e., $m=\dfrac{1}{N}\sum_{i=1}^N x_i$, $F:\mathcal{X}\rightarrow \mathcal{X}$  is the mean field coupling term, and $G:\mathcal{X}\rightarrow\R$ is the cost associated with the mean field term. We impose the following regularity conditions on $J_i$:
\begin{assumption}
\label{assumption_cost}
(i) $V_i(x_i,u_i)$ is convex with respect to  $(x_i,u_i)$,
(ii)  $F(\cdot)$  is globally Lipschitz continuous with constant $L$ on $\mathcal{X}$, 
(iii)  $G(m)$ is Fr\'echet differentiable, and the derivative of $G(m)$ is bounded for any bounded $m$. 
\end{assumption}

The mean field game can be then formulated as follows:
\begin{align}
\label{individualoptimization}
&\min_{u_i} \mathbb{E} \left(V_i(x_i,u_i)+F(m)\cdot  x_i +G(m)\right) \\
&\text{s.t.} 
\begin{cases}
\label{abstractdynamic}
x_i=f_i(u_i, \pi_i) \\
m=\dfrac{1}{N}\sum_{i=1}^N x_i,  x_i\in \mathcal{X}_i,\quad u_i\in \bar{\mathcal{U}}_i.
\end{cases}
\end{align}

 There are several solution concepts for the game problem, such as Nash equilibrium, Bayesian Nash equilibrium, dominant strategy equilibrium, among others. In the context of mean field games, we usually relax the Nash equilibrium solution concept by assuming that each agent is indifferent to an arbitrarily small change $\epsilon$. This solution concept is referred to as the $\epsilon$-Nash equilibrium, formally defined as follows:
\begin{definition}
$(u_1^*,\ldots,u_N^*)$ is an $\epsilon$-Nash equilibrium of the game (\ref{individualoptimization}) if  the following inequality holds 
\begin{equation}
\mathbb{E} J_i(u_i^*,u_{-i}^*)\leq \mathbb{E} J_i(u_i,u_{-i}^*)+\epsilon
\end{equation}
for all $i=1,\ldots,N$, and all $u_i\in \bar{\mathcal{U}}_i$, where $u_{-i}=(u_1,\ldots,u_{i-1},u_{i+1},\ldots,u_N)$ and $J_i(u_i,u_{-i})$ is the compact notation for (\ref{utilityfunction}) after plugging (\ref{systemdynamics}) in (\ref{utilityfunction}).
\end{definition}

At an $\epsilon$-Nash equilibrium, each agent can lower his cost by at most $\epsilon$ via deviating from the equilibrium strategy, given that all other players follow the equilibrium strategy. 

\begin{remark}
The game problem (\ref{individualoptimization}) formulates a broad class of mean field games. As the problem is defined in the general functional space, it provides a unifying formulation that includes both discrete-time \cite{ma2013decentralized}, \cite{rammatico2016decentralized} and continuous-time system \cite{huang2007large} as special cases, and addresses both deterministic and stochastic cases. This class of problems frequently arises in various applications \cite{huang2007large}, \cite{gueant2011mean}, \cite{couillet2012electrical}, \cite{ma2013decentralized}, \cite{rammatico2016decentralized},   \cite{huang2012social},    \cite{huang2010nce},  \cite{kamgarpour2013bayesian},  \cite{mfgoil}, where $F(m)\cdot x_i$ can be either interpreted as the price multiplied by quantity \cite{lachapelle2010computation}, \cite{li2016on} or part of the quadratic penalty of the deviation of the system state from the population mean \cite{huang2007large}, \cite{huang2012social}. 
\end{remark}

\subsection{Examples}
This subsection presents two examples of mean field games. One is a discrete-time deterministic game \cite{rammatico2016decentralized} and the other is a continuous-time stochastic game \cite{huang2007large}. We show that these two examples are special cases of the formulated mean field game problem (\ref{individualoptimization}).

\subsubsection{Discrete-time deterministic game}
In \cite{rammatico2016decentralized} a deterministic mean field game is formulated in discrete-time as follows:
\begin{align}
\label{examplelqg1}
&\min_{\{u_i(t),t\geq 1\}} \sum_{t=1}^K \left(\left||x_i(t)-v(t) \right||^2+\left|| u_i(t)\right||^2 \right)\\
&\text{s.t.}
\begin{cases}
\label{linearsde1}
x_i(t)=\alpha_ix_i(t-1)+\beta_i u_i(t) \\
x_i(t)\in \mathcal{X}_i^t,\quad u_i(t)\in \bar{\mathcal{U}}_i^t, \quad i=1,\ldots,N,
\end{cases}
\end{align}
where $x_i(t)$ and $u_i(t)$ denote the state and control of agent $i$ at time $t$, and $v_i$ is the mean field term that satisfies $v(t)= \gamma \dfrac{1}{N} \sum_{i=1}^N x_i(t)+\eta $. To transform the above game problem to (\ref{individualoptimization}), we stack the control and state at different times to form the state and control trajectories, i.e., $x_i=\{x_i(t),t\in T\}$ and $u_i=\{u_i(t), t\in T\}$, where $T=\{1,\ldots,K\}$. Similarly, define $v=\{v(t),t\in T\}$.  Then we have $\mathcal{X}_i=\mathcal{X}_i^1\times \cdots \times \mathcal{X}_i^K$ and $\mathcal{U}_i=\mathcal{U}_i^1\times \cdots \times \mathcal{U}_i^K$, which are all Euclidean spaces. In this case, the discrete time game problem can be transformed to the following form:
\begin{align}
\label{gameform1}
&\min_{u_i} ||x_i||^2+||v||^2-2x_i\cdot v+||u_i||^2.\\
&\text{s.t.} 
\begin{cases}
\label{linearsde1}
x_{i}=C u_i +D \\
x_{i}\in \mathcal{X}_i,\quad u_{i}\in \bar{\mathcal{U}}_i, \quad i=1,\ldots,N,
\end{cases}
\end{align}
where $C$ is a matrix that depends on $\alpha_i$ and $\beta_i$, and $D$ is a constant vector. It is easy to verify that the game problem (\ref{gameform1}) satisfies Assumption \ref{assumptionf_i}-\ref{assumption_cost}. Therefore, (\ref{gameform1}) is a special case of our proposed mean field game (\ref{individualoptimization}).

\subsubsection{Continuous-time stochastic game}
The second example involves a linear quadratic Gaussian (LQG) control problem formulated in \cite{huang2007large}. The game problem can be written as follows:
\begin{align}
\label{examplelqg}
&\min_{\{u_i(t),t\geq 0\}} \mathbb{E} \int_0^\infty e^{-\rho t}\big[(x_i(t)-v(t))^2+ru_i(t)^2\big]dt \\
&\text{s.t.}
\begin{cases}
\label{linearsde}
dx_i(t)=(\alpha_ix_i(t)+\beta_i u_i(t))dt+\sigma_i dB_i(t) \\
x_i(t)\in \mathbb{R},\quad u_i(t)\in \mathbb{R}, \quad i=1,\ldots,N,
\end{cases}
\end{align}
where $x_i(t)$ and $u_i(t)$ denote the state and control for the $i$th agent at time $t$, $B_i(t)$ is a standard scalar Brownian motion, $v(t)=\gamma \dfrac{1}{N} \sum_{i=1}^N x_i(t)+\eta$ is the mean field term, and $\rho, r,\gamma, \eta>0$ are real constants. Let $x_i=\{x_i(t),t\geq 0\}$ and $u_i=\{u_i(t),t\geq 0\}$ denote the state and the control, respectively. The admissible control set is then defined as $\bar{\mathcal{U}}_i=\{u_i\in \mathcal{U}_i \vert u_i\in C[0,\infty),x_i=f_i(u_i,\pi_i), \mathbb{E} ||x_i||^2\leq C\}$, where $f_i$ can be defined based on the following relation:
\begin{align}
\label{f_ilqg}
x_i(t)=&x_i(0)e^{\alpha_i t}+\int_{0}^{t}\beta_i e^{\alpha_i(t-s)}u_i(t)dt \nonumber \\
&+\sigma_i\int_{0}^{t}e^{\alpha_i(t-s)}dB_i(s).
\end{align}
In this case, the state space satisfies $\mathcal{X}_i=\mathcal{X}$, which consists of all real-valued functions that are square integrable after multiplying a discounting factor, $\mathcal{X}=\{x|x(t)\in \mathbb{R}, \int_{[0,\infty)}e^{-\rho t}|x(t)|^2dt<\infty\}$.  For each $x_i$ and $x_j$ in $\mathcal{X}$, we define the inner product of $x_i$ and $x_j$ as follows:
\begin{equation}
x_i\cdot x_j=\int_{[0,\infty)} e^{-\rho t}x_i(t)x_j(t)dt,
\end{equation}
and the norm is defined as $||x_i||=\sqrt{x_i\cdot x_i}$. Note that there is an isometry mapping 
 between the state space $\mathcal{X}$ and the $L^2$ space, and therefore, $\mathcal{X}$ is complete. This gives the following lemma:
\begin{lemma}
\label{hilbertspace}
The state space $\mathcal{X}$ is a Hilbert space.  
\end{lemma}
\begin{proof}
Since $L^2$ space is complete, we show that $\mathcal{X}$ is isomorphic to $L^2$, and therefore it is also complete \cite[p.20]{conway2013course}. For this purpose, we define the following mapping: $g:L^2\rightarrow \mathcal{X}$ that satisfies $g(\cdot)=\{g_t(\cdot),t\in T\}$ and $g_t(l(t))=e^{\rho t/2}l(t)$ for any $l(t)\in L^2$. This is a linear surjective mapping and it can be verified that $g(l_1(t))\cdot g(l_2(t))=l_1(t)\cdot l_2(t)$, where the left-hand side inner product is defined on $\mathcal{X}$ and the right-hand side inner product is defined on the $L^2$ space. This indicates that $L^2$ and $\mathcal{X}$ are isomorphic, which completes the proof. 
\end{proof}

\begin{remark}
\label{completeness}
The completeness of the state space is important for future discussions. We will elaborate more on this point in later sections.
\end{remark}

Under the inner product defined above, the objective function of the problem (\ref{examplelqg}) can be transformed to the form of (\ref{individualoptimization}):
\begin{equation}
\label{transformedcont}
V_i(x_i,u_i)+F(m)\cdot x_i+G(m),
\end{equation}
where $V_i(x_i,u_i)=||x_i||^2+||u_i||^2$, $F(m)=-2m$ and $G(m)=||m||^2$.

It is easy to verify Assumption \ref{assumptionf_i} in the above example. To verify Assumption \ref{assumption_convex}, we note that $V_i(f_i(u_i,\pi_i),u_i)$ is convex with respect to $(f_i(u_i,\pi_i),u_i)$, and according to (\ref{f_ilqg}), $x_i=f_i(u_i,\pi_i)$ is an affine operator  with respect to $u_i$. This indicates that $V_i(f_i(u_i,\pi_i),u_i)$ is convex with respect to $u_i$. In addition,  according to $(\ref{transformedcont})$,   it is clear that $F(m)$ is Lipschitz continuous,  and $G(m)$ is Fr\'echet differentiable with bounded derivative for a bounded $m$. This validates Assumption \ref{assumption_convex}.

To verify Assumption \ref{assumption_cost}, we note that $x_i$ is an affine operator with respect to both $u_i$ and $\pi_i$, which conveniently leads to the convexity of $\bar{\mathcal{U}}_i$. This result can be summarized in the following lemma:
\begin{lemma}
\label{proofassumption2}
The admmissible control set $\bar{\mathcal{U}}_i$ for the game problem (\ref{examplelqg}) is convex. 
\end{lemma}

To summarize, the game problem in \cite{huang2007large} satisfies all the assumptions and can be viewed as a special case of the proposed mean field game (\ref{individualoptimization}).

\subsection{The Social Welfare}
The game equilibrium characterizes the outcomes of the interaction of agents with conflicting objectives. At the equilibrium solution, the conflicts are resolved and each agent achieves individual cost minimization. However, when it comes to evaluating the system-level performance, we run into another hierarchy of conflict: the conflict between the cost of individual agents and the system-level objective. For this problem, an important question is under what condition the game equilibrium is to the best benefit of the overall system, and is therefore desirable from the system level.  We interpret the system-level objective by the following social welfare optimization problem:
\begin{align}
\label{eq:socialwelfaremax}
&\min_{u_1,\ldots,u_N} \mathbb{E} \left(\sum_{i=1}^N V_i(x_i,u_i)+\phi(\mathbb{E} m)\right)\\
&\text{s.t.}
\begin{cases}
x_i=f_i(u_i, \pi_i) \\
x_i\in \mathcal{X}_i,\quad u_i\in \bar{\mathcal{U}}_i, \quad \forall i=1,\ldots,N
\end{cases}
\end{align}
where $\phi:\mathcal{X}\rightarrow \mathbb{R}$ represents the cost for the overall system to attain the average state. We assume it to be convex and Fr\'echet differentiable. 

In general, due to the self-interest seeking nature of the individual agents, the equilibrium of the mean field game (\ref{individualoptimization}) and the solution to the social welfare optimization problem (\ref{eq:socialwelfaremax}) are inherently inconsistent.  When the mean field equilibrium coincides with the solution to the social welfare optimization problem (\ref{eq:socialwelfaremax}), we say the game equilibrium is socially optimal. 
The objective of this paper is to study the $\epsilon$-Nash equilibrium of the mean field game (\ref{individualoptimization}), and investigate the conditions under which the game equilibrium is socially optimal.

\section{Characterizing The $\epsilon$-Nash Equilibrium}
In this section, we characterize the $\epsilon$-Nash equilibrium of the mean field game (\ref{individualoptimization}) with a set of mean field equations. Our derivation follows similar ideas in \cite{huang2007large} and \cite{rammatico2016decentralized}. However, as we consider the mean field game in general functional spaces, our result applies to more general cases than linear quadratic problems in \cite{huang2007large} and \cite{rammatico2016decentralized}. More comparisons can be found in Remark \ref{comparison}.

To study mean field equilibrium, we note that the cost function (\ref{utilityfunction}) of the individual agent is only coupled through the mean field term $F(m)$ and $G(m)$, and the impact of the control input for a single agent on the coupling term vanishes as the population size goes to infinity. Therefore, in the large population case we can approximate the agent's behavior with the optimal response to a deterministic value $y\in \mathcal{X}$ that replaces the mean field term $F(m)$ in the cost function (\ref{utilityfunction}). For this purpose, we define the following optimal response problem:
\begin{align}
\label{optimalresponse}
\mu_i(y)=&\argmin_{u_i} \mathbb{E} \left(V_i(x_i,u_i)+y\cdot  x_i \right) \\
&\text{s.t.} 
\begin{cases}
x_i=f_i(u_i, \pi_i) \\
x_i\in \mathcal{X}_i,\quad u_i\in \bar{\mathcal{U}}_i,
\end{cases}
\end{align}
where $\mu_i(y)$ denotes the optimal solution to the optimization problem (\ref{optimalresponse}) parameterized by $y$, and $G(m)$ is regarded as a constant in (\ref{optimalresponse}) that can be ignored.  Based on this approximation,  $y$  generates a collection of agent responses. Ideally, the deterministic mean field term approximation $y$ guides the individual agents to choose a collection of optimal responses $\mu_i(y)$ which, in return, collectively generate the mean field term that is close to the approximation $y$. This suggests that we use the following equation systems to characterize the equilibrium of the mean field game:
\begin{numcases}{}
   \mu_i(y)=\argmin_{u_i\in \bar{\mathcal{U}}_i} \mathbb{E} \left(V_i(x_i,u_i)+y\cdot  x_i \right)      \label{equationsystem1}
   \\
   x_i^*=f_i(\mu_i(y), \pi_i)  \label{equationsystem2} \\
   y=F\left(\dfrac{1}{N}\mathbb{E}\sum_{i=1}^N x_i^*\right),  \label{equationsystem3}
\end{numcases}
where the mean field term $y$ induces a collection of responses that generate $y$. In the rest of this subsection we show that the solution to this equation system is an $\epsilon$-Nash equilibrium of the mean field game (\ref{individualoptimization}). For this purpose, we first prove the following lemmas:
\begin{lemma}
\label{lemma1}
Under Assumption \ref{assumptionf_i}(i), if there exists $C>0$ such that $\mathbb{E}||x_i||^2\leq C$ for all $i=1,\ldots,N$, and $F(\cdot)$ is globally Lipschitz continuous, then the following relation holds for each agent $i$:
\begin{equation}
\label{prooflemma}
\big|\mathbb{E} \big( F(m)\cdot x_i\big)-F(\mathbb{E}m)\cdot \mathbb{E} x_i \big|\leq \epsilon_{N},
\end{equation}
where $m=\dfrac{1}{N}\sum_{i=1}^N x_i$ and $\lim_{N\rightarrow \infty} \epsilon_{N}=0$. 
\end{lemma}
\vspace{0.2cm}
\begin{proof}
To prove this result, it is clear that we have:
\begin{align*}
\big|\mathbb{E} \big(F(m)\cdot x_i\big)-&F(\mathbb{E}m)\cdot \mathbb{E} x_i \big|=\big| I_1+I_2 \big| \leq \nonumber\\ 
&\big| I_1\big| +\big| I_2\big|,
\end{align*}
where we define $I_1$ as $I_1=\mathbb{E} \big(F(m)\cdot x_i\big)-\mathbb{E} F(m)\cdot\mathbb{E} x_i$ and $I_2=\mathbb{E} F(m)\cdot\mathbb{E} x_i-F(\mathbb{E}m)\cdot \mathbb{E} x_i$. To show that $\big| I_1\big|$ is bounded by some $\epsilon$, we notice that:
\begin{align*}
\big| I_1\big|=\big| \mathbb{E} \big(F(m)\cdot x_i\big)-&\mathbb{E} F(m)\cdot\mathbb{E} x_i \big|=\big| I_3+I_4\big| \leq \nonumber \\
&\big| I_3\big| +\big| I_4\big|,
\end{align*}
where we define $I_3=\mathbb{E} \big(F(m)\cdot x_i\big)-\mathbb{E}\big(F(m_{-i})\cdot x_i\big)$, $I_4=\mathbb{E} \big(F(m_{-i})\cdot x_i\big)-\mathbb{E}\big(F(m)\cdot \mathbb{E}x_i\big)$ and $m_{-i}=\dfrac{1}{N}\sum_{j\neq i}x_j$.
Since $F(\cdot)$ is Lipschitz continuous with the constant $L\geq 0$, and the second moment of $x_i$ is bounded, we have:
\begin{align*}
\big| I_3\big|=\big| \mathbb{E}& \big(F(m)\cdot x_i\big)-\mathbb{E}\big(F(m_{-i})\cdot x_i\big) \big| \leq \nonumber\\
\mathbb{E}\big| F(m)\cdot x_i&-F(m_{-i})\cdot x_i\big| \leq \mathbb{E} \big(\big\Vert \dfrac{L}{N} x_i\big\Vert \big\Vert x_i  \big\Vert \big) = \nonumber \\
&\dfrac{L}{N}\mathbb{E}||x_i||^2 \leq \dfrac{LC}{N}. 
\end{align*}
In addition, as $x_i$ is independent with $m_{-i}$, we have 
\begin{align*}
\big| I_4\big|=\big| \mathbb{E} \big(F(m_{-i})\cdot x_i\big)-\mathbb{E}\big(F(m)\cdot \mathbb{E}x_i\big) \big| = \nonumber\\
\big|\mathbb{E} F(m_{-i})\cdot \mathbb{E}x_i-\mathbb{E}F(m)\cdot \mathbb{E}x_i \big| \leq \dfrac{L}{N}\big\Vert \mathbb{E}x_i\big\Vert^2. 
\end{align*}
Note that $\mathbb{E}||x_i||^2$ is bounded, and thus $\big\Vert\mathbb{E}x_i\big\Vert$ is also bounded:
\begin{equation*}
\big\Vert \mathbb{E}x_i\big\Vert\leq \mathbb{E} \big\Vert x_i\big\Vert =\mathbb{E} \sqrt{||x_i||^2} \leq \sqrt{\mathbb{E}||x_i||^2}\leq \sqrt{C}.
\end{equation*}
This indicates that $\big| I_4\big|\leq \dfrac{LC}{N}$.
 Therefore, $\big| I_1\big|\leq \big| I_3\big|+\big| I_4\big|\leq \dfrac{2LC}{N}$. To show that $\big| I_2\big|$ converges to $0$, we define a random variable $r_N=F(m)-F(\mathbb{E}m)$. 
Since $x_i$ and $x_j$ are independent and have bounded second moment,  by the $L_2$ weak law of large numbers,  $m$ converges to $\mathbb{E}m$ in $L_2$ as $N$ goes to infinity (can be easily proved in infinite-dimensional case). Therefore, $m$ converges to $\mathbb{E}m$ in $L_1$, i.e., $\lim_{N\rightarrow \infty}\mathbb{E}||m-\mathbb{E}m|| =0$. In addition, we have:
\begin{equation*}
||r_N||=||F(m)-F(\mathbb{E}m)||\leq L||m-\mathbb{E}m||.
\end{equation*}
Therefore, $\lim_{N\rightarrow \infty}\mathbb{E}||r_N||\leq \lim_{N\rightarrow \infty}\mathbb{E}L||m-\mathbb{E}m|| =0$
This indicates that:
\begin{equation*}
I_2\leq \big\Vert \mathbb{E}F(m)-F(\mathbb{E}m)\big\Vert\big\Vert \mathbb{E}x_i \big\Vert\leq \sqrt{C}\epsilon_N.
\end{equation*} 
This completes the proof.
\end{proof}

\begin{lemma}
\label{lemma2}
Define $m=\dfrac{1}{N} \sum_{i=1}^Nx_i$, and let $m_{-i}=\dfrac{1}{N} \sum_{j\neq i} x_j$, where $x_i$ denotes the state trajectory corresponding to $u_i$ in $\bar{\mathcal{U}_i}$, then we have the following relation:
\begin{equation}
\big |\mathbb{E}G(m)-\mathbb{E}G(m_{-i})\big|\leq \epsilon_N
\end{equation}
\end{lemma}
\vspace{0.1cm}
\begin{proof}
For notation convenience, we use $\epsilon_{iN}$ to denote a sequence that converges to 0 for any fixed $i$ as $N$ goes to infinity, i.e., $\lim_{N\rightarrow\infty}\epsilon_{iN}=0$, where $i=1,2,\ldots$. To prove this lemma, we can show that  $|\mathbb{E}G(m)-G(\mathbb{E}m)|\leq \epsilon_{1N}$ and  $|\mathbb{E}G(m_{-i})-G(\mathbb{E}m_{-i})|\leq \epsilon_{2N}$ using the same argument as in the proof of Lemma \ref{lemma1}.  Then it suffices to show that $|G(\mathbb{E}m)-G(\mathbb{E}m_{-i})|\leq \epsilon_{3N}$. This is true since $||\mathbb{E}m||$ is bounded in a compact set, and  $G(\mathbb{E}m)$ has bounded derivative for any bounded $\mathbb{E}m$.    
\end{proof}

Using the result of Lemma \ref{lemma1} and Lemma \ref{lemma2}, we can show that the solution to  the equation system (\ref{equationsystem1})-(\ref{equationsystem3}) is an $\epsilon$-Nash equilibrium of the mean field game (\ref{individualoptimization}), which is summarized in the following theorem.
\begin{theorem}
\label{theorem1}
The solution to the equation system (\ref{equationsystem1})-(\ref{equationsystem3}), is an $\epsilon_N$-Nash equilibrium of the mean field game (\ref{individualoptimization}), and  $\lim_{N\rightarrow \infty} \epsilon_N=0$. 
\end{theorem}
\begin{proof}
For notation convenience, we denote the solution to the equation system (\ref{equationsystem1})-(\ref{equationsystem3}) as $u_i^*$, $x_i^*$ and $y^*$, where $x_i^*$ is the state trajectory corresponding to $u_i^*$. 
To prove this theorem, we need to show that:
\begin{align*}
&\mathbb{E} \bigg( V_i(x_i^*,u_i^*)+F(m^*)\cdot x_i^*+G(m^*)\bigg) \leq \epsilon_N+\nonumber\\ \mathbb{E} \bigg(&V_i(x_i,u_i)+F\bigg(\dfrac{1}{N}x_i+m_{-i}^* \bigg)\cdot x_i+G\bigg(\dfrac{1}{N}x_i+m_{-i}^*\bigg)\bigg)
\end{align*}
for all $u_i\in \bar{\mathcal{U}}_i$, where $m^*=\dfrac{1}{N}\sum_{i=1}^N x_i^*$, $m_{-i}^*=\dfrac{1}{N}\sum_{j\neq i} x_j^*$, and $x_i$ is the state trajectory corresponding to $u_i$. 
Based on Lemma \ref{lemma1} and Lemma \ref{lemma2}, it suffices to show that:
\begin{align}
\label{provecondition}
&\mathbb{E} V_i(x_i^*,u_i^*)+F\bigg(\dfrac{1}{N}\mathbb{E}\sum_{i=1}^N x_i^*\bigg)\cdot \mathbb{E} x_i^* \leq \nonumber\\ \mathbb{E}V_i(x_i,&u_i)+F\bigg(\dfrac{1}{N}\mathbb{E}\big(x_i+\sum_{j\neq i}x_j^* \big)\bigg)\cdot \mathbb{E}x_i+\epsilon_{1N}.
\end{align}
Since $||\mathbb{E} x_i||$ is bounded (see proof for Lemma \ref{lemma1}) and $F(\cdot)$ is Lipschitz continuous with  constant $L\geq 0$, we have:
\begin{align}
\label{provecondition_temp}
\bigg|F\bigg(\dfrac{1}{N}\mathbb{E}(x_i+&\sum_{j\neq i}x_j^* )\bigg)\cdot \mathbb{E}x_i-F\bigg(\dfrac{1}{N}\mathbb{E}\sum_{i=1}^N x_i^*\bigg)\cdot \mathbb{E}x_i\bigg|\leq \nonumber \\
&\big\Vert \dfrac{1}{N}L(\mathbb{E}x_i-\mathbb{E}x_i^*)\big\Vert \big\Vert \mathbb{E}x_i \big\Vert=\epsilon_{2N}.
\end{align}
Therefore, combining (\ref{provecondition}) and (\ref{provecondition_temp}), it suffices to show that:
\begin{align*}
\mathbb{E} V_i(x_i^*,&u_i^*)+F\bigg(\dfrac{1}{N}\mathbb{E}\sum_{i=1}^N x_i^*\bigg)\cdot \mathbb{E} x_i^* \leq \nonumber\\ &\mathbb{E}V_i(x_i,u_i)+F\bigg(\dfrac{1}{N}\mathbb{E}\sum_{i=1}^N x_i^*\bigg)\cdot \mathbb{E}x_i+\epsilon_{3N}.
\end{align*}
Note that based on (\ref{equationsystem3}), $F\bigg(\dfrac{1}{N}\mathbb{E}\sum_{i=1}^N x_i^*\bigg)=y^*$, which is equivalent to:
\begin{equation*}
\mathbb{E} V_i(x_i^*,u_i^*)+y^*\cdot \mathbb{E} x_i^* \leq \mathbb{E}V_i(x_i,u_i)+y^*\cdot \mathbb{E}x_i+\epsilon_{3N}.
\end{equation*}
This obviously holds based on (\ref{equationsystem1}), which completes the proof.
\end{proof}

Theorem \ref{theorem1} indicates that each agent is motivated to follow the equilibrium strategy $u_i^*$ as deviating from this strategy can only decrease the individual cost by a negligible amount $\epsilon_N$. Furthermore, this $\epsilon_N$ can be arbitrarily small, if the population size is sufficiently large.

\begin{remark}
\label{comparison}
Our result in Theorem \ref{theorem1} generalizes existing works in (\cite{huang2007large} and \cite{rammatico2016decentralized}) from several perspectives. First, these works mainly focus on linear quadratic problem, while we consider a general mean field game formulated in the functional space. Therefore, our result applies to more general cases than quadratic individual costs. Furthermore, the mean field coupling term $F(\cdot)$ is assumed to be affine  in  \cite{huang2007large} and \cite{rammatico2016decentralized}, while we relax it to be Lipschitz continuous. In the affine case, the order of expectation and the function $F(\cdot)$ can be interchanged, which significantly  simplifies the proof of Lemma \ref{lemma1}. On the other hand, our generalization also comes at a cost: since $F(\cdot)$ is Lipschitz continuous, we can not derive the convergence rate of $\epsilon$ as in \cite{huang2007large}.
\end{remark}

\section{Connection to Social Welfare Optimization}
This section tries to connect the mean field game (\ref{individualoptimization}) to the social welfare optimization problem (\ref{eq:socialwelfaremax}). Since the social welfare (\ref{eq:socialwelfaremax}) involves the sum of individual cost functions, the key idea is to decompose the social welfare optimization problem into individual cost minimization problems. To this end, we note that while the first term of (\ref{eq:socialwelfaremax}) is separable over $u_i$, the second term $\phi(\cdot)$ couples the decisions of each agent and makes the decomposition difficult. To address this issue, we introduce an augmented decision variable $z$ to denote the average of the expectation of the population state, i.e., $z=\dfrac{1}{N}\mathbb{E} \sum_{i=1}^N x_i$. The social welfare maximization problem (\ref{eq:socialwelfaremax}) can be then transformed to the following form:.  
\begin{equation}  
\label{socialwelfaretransform}
\min_{u_1,\ldots,u_N,z} \mathbb{E} \left(\sum_{i=1}^N V_i(x_i,u_i)+\phi(z)\right)
\end{equation}
\begin{subnumcases}{s.t.}
\label{constraint2}
z=\dfrac{1}{N}\mathbb{E} \sum_{i=1}^N x_i \\
x_i=f_i(u_i,\pi_i), \forall i=1,\ldots,N \\
x_i\in \mathcal{X}_i,\quad u_i\in \bar{\mathcal{U}}_i,\quad \forall i=1,\ldots,N
\end{subnumcases}
It is obvious that this problem (\ref{socialwelfaretransform}) and the social welfare optimization problem (\ref{eq:socialwelfaremax}) are equivalent, but the advantage of this transformation is that the objective function is now separable over its decision variables $(u_1,\ldots,u_N,z)$, and dual decomposition can be applied to decompose the problem into individual cost minimization problems. This can be done by incorporating the constraint (\ref{constraint2}) in the Lagrangian of the objective function, and obtaining the following equation:
\begin{align}
\label{lagrangianequation}
L(u,z,\lambda)=& \sum_{i=1}^{N} \mathbb{E} V_i(x_i,u_i)+\phi(z) \nonumber \\
&+ \lambda \cdot \left(\mathbb{E} \sum_{i=1}^N x_i-Nz\right) 
\end{align}
where $u=(u_1,\ldots,u_N)$ and $x=(x_1,\ldots,x_N)$. Note that the completeness of $\mathcal{X}$ is important in this derivation. If $\mathcal{X}$ is not complete, then the inner product term in (\ref{lagrangianequation}) should be replaced by a linear operator. In this case, since the mean field equations have inner product term, they can not be connected to the socially optimal solution. This elaborates Remark \ref{completeness}.

 For notation convenience, we define $L_i(u_i,\lambda)=\mathbb{E} V_i(x_i,u_i)+\lambda\cdot \mathbb{E} x_i$ to be the Lagrangian for the $i$th agent, and denote $L_0(z,\lambda)=  \phi (z)-N\lambda\cdot z$ as the Lagrangian for a virtual agent (social planner), then the dual problem of (\ref{socialwelfaretransform}) can be written as follows:
\begin{equation}  
\label{dualprolem}
\max_{\lambda} \min_{u_1,\ldots, u_N,z} \sum_{i=1}^N L_i(u_i,\lambda)+L_0(z,\lambda),
\end{equation}
\begin{subnumcases}{s.t.}
\label{constraint3}
x_i=f_i(u_i, \pi_i), \forall i=1,\ldots,N \\
x_i\in \mathcal{X}_i,\quad u_i\in \bar{\mathcal{U}}_i, \forall i=1,\ldots,N
\end{subnumcases}
Since the dual problem is fully decomposible, we can decompose it into individual problems so as to connect the dual solution to the mean field equilibrium. Furthermore, if the duality gap of the social welfare optimization problem is zero, then the mean field equilibrium is socially optimal. The main result of this paper can be summarized as follows:

\begin{theorem}
\label{connectiontheorem1}
Let $\phi(\cdot)$ satisfies $F(z)=\dfrac{1}{N} \phi'(z)$ for $\forall z\in \mathcal{X}$, and assume that the Slater condition holds for the social welfare optimization problem (\ref{socialwelfaretransform}), i.e., $\bar{\mathcal{U}}_i$ has at least one interior point, then any solution to the mean field equations (\ref{equationsystem1})-(\ref{equationsystem3}) is also the solution to the social welfare optimization problem (\ref{socialwelfaretransform}), and vice versa. 
\end{theorem} 
\begin{proof}
We first note that since the social welfare optimization problem is convex, and the Slater condition holds, then the duality gap between (\ref{socialwelfaretransform}) and (\ref{dualprolem}) is zero, and there exists a dual solution to (\ref{dualprolem}), denoted as $(\lambda^d,u_1^d,\ldots,u_N^d,z^d)$, such that primal constraints are satisfied \cite[p. 224]{luenberger1997optimization}, i.e., $z^d=\dfrac{1}{N}\mathbb{E}\sum_{i=1}^N x_i^d$. To prove the theorem, if suffices to show that there is a one-to-one correspondence between the solution to (\ref{dualprolem}) and the solution to the mean field equations (\ref{equationsystem1})-(\ref{equationsystem3}). Since $(\lambda^d,u_1^d,\ldots,u_N^d,z^d)$ is the optimal solution to (\ref{dualprolem}), the following relations holds:
\begin{align}
\label{dualcondition1}
u_i^d= &\argmin_{u_i} L_i(u_i,\lambda^d)  \\
&\text{s.t.} 
\begin{cases}
x_i=f_i(u_i,\pi_i) \\
x_i\in \mathcal{X}_i,\quad u_i\in \bar{\mathcal{U}}_i
\end{cases}
\end{align}
and we also have:
\begin{equation}
\label{dualcondition2}
z^d=\argmin_{z\in \mathcal{X}}  L_0(z,\lambda^d) 
\end{equation}
Based on the definition (\ref{optimalresponse}), the equation (\ref{dualcondition1}) is equivalent to $u_i^d=\mu_i(\lambda^d)$, which is the same as equation (\ref{equationsystem1}). As $\phi(\cdot)$ is convex, we take the derivative of the objective function in (\ref{dualcondition2}) to be 0 and then (\ref{dualcondition2}) is equivalent to $\lambda^d=\dfrac{1}{N}\phi(z^d)'=F(z^d)$, where the second equality is due to the assumption of this theorem.
Note that at optimal solution, we have $z^d=\dfrac{1}{N}\mathbb{E}\sum_{i=1}^N x_i^d$. Therefore, $\lambda^d=F\left(\dfrac{1}{N}\mathbb{E} (\sum_{i=1}^N   x_i^d )\right)$, which is equivalent to (\ref{equationsystem1}). This indicates that the optimality condition of the dual problem is equivalent to the mean field equations (\ref{equationsystem1})-(\ref{equationsystem3}). This completes the proof.
\end{proof}

Theorem \ref{connectiontheorem1} establishes the connection between the mean field equilibrium and the social welfare optimization problem. This connection is important in several perspectives. First, it enables to evaluate the system-level performance of the mean field game. In particular, when a mean field game is given, we know under what condition the game equilibrium is socially optimal. Second, our result indicates that computing the mean field equilibrium is essentially equivalent to solving a convex social welfare optimization problem. Therefore, we can compute the mean field equilibrium by considering the corresponding social welfare optimization problem, which can be efficiently solved by various convex optimization methods. 


\begin{remark}
\label{remarkcontrol}
In this paper, we formulate the mean field term $F(\cdot)$ to depend on the average of the population {\bf state}. However, the proposed method still works when the mean field is the average of the {\bf control} decisions. In this case, the individual cost (\ref{utilityfunction}) is defined as $J_i(x_i,u_i,F(m), G(m))=V_i(x_i,u_i)+F\bigg(\dfrac{1}{N}\sum_{i=1}^N u_i\bigg)\cdot u_i+G\bigg(\dfrac{1}{N}\sum_{i=1}^N u_i\bigg)$, and the constraint (\ref{constraint2}) is $z=\dfrac{1}{N}\sum_{i=1}^N u_i$, then the same result as Theorem \ref{connectiontheorem1} can be obtained by a similar approach.
\end{remark}

\begin{figure*}[bt]%
\begin{minipage}[b]{0.32\linewidth}
\centering
\includegraphics[width = 1\linewidth]{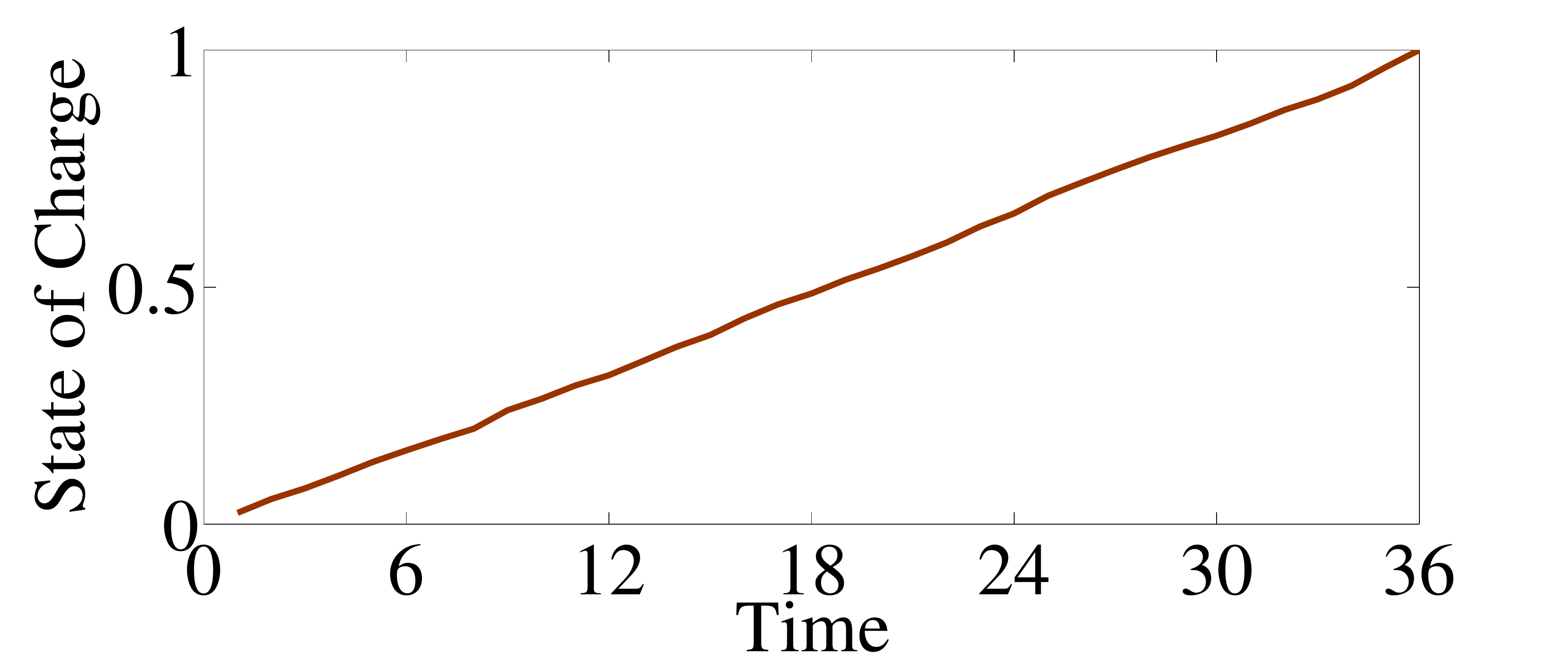}
\caption{The SOC of a randomly selected electric vehicle under ADMM.}
\label{fig1}
\end{minipage}
\begin{minipage}[b]{0.01\linewidth}
\hfill
\end{minipage}
\begin{minipage}[b]{0.32\linewidth}
\centering
\includegraphics[width = 1\linewidth]{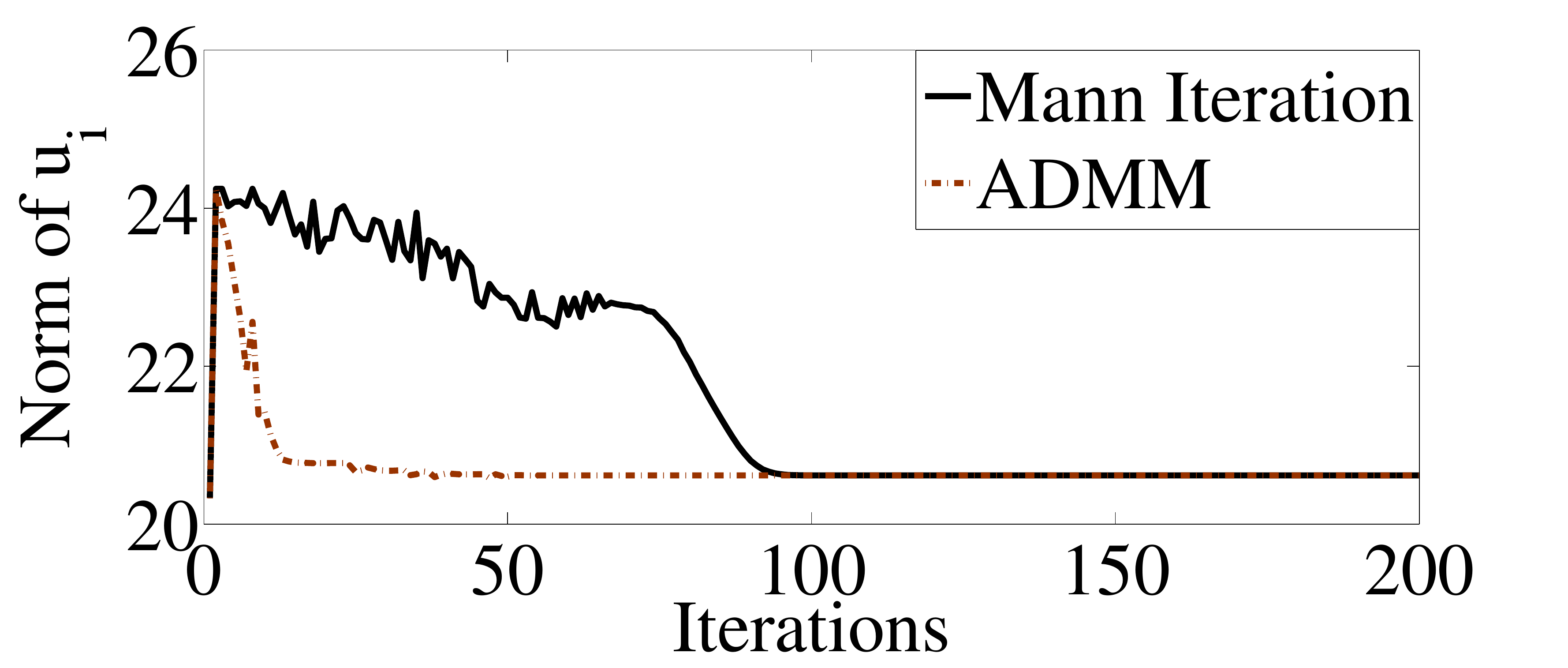}
\caption{The control decision of a randomly selected electric vehicle.}
\label{fig2}
\end{minipage}
\begin{minipage}[b]{0.01\linewidth}
\hfill
\end{minipage}
\begin{minipage}[b]{0.32\linewidth}
\centering
\includegraphics[width = 0.99\linewidth]{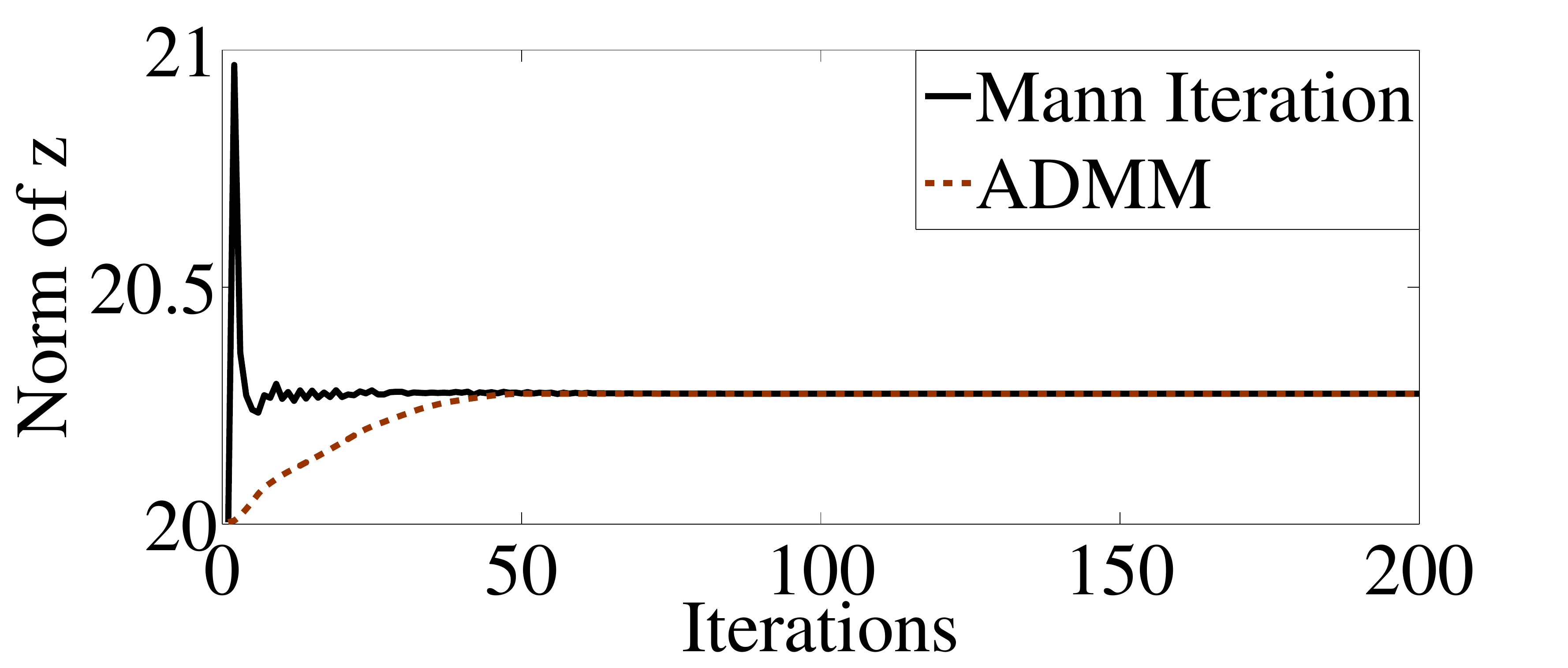}
\caption{The average control decision of the population in two algorithms.}
\label{fig3}
\end{minipage}
\end{figure*}

\section{Computing the mean field equilibrium via primal-dual}
In this section we propose a decentralized algorithm to compute the mean field equilibrium based on Theorem \ref{connectiontheorem1}. 
When a mean field game (\ref{individualoptimization}) is given, instead of directly solving the equation systems (\ref{equationsystem1})-(\ref{equationsystem3}), we construct a social welfare maximization problem (\ref{eq:socialwelfaremax}) by finding a cost functional $\phi(\cdot)$ such that $\phi(z)'=NF(z)$. According to Theorem~\ref{connectiontheorem1}, the optimal solution to the social welfare maximization problem (\ref{eq:socialwelfaremax}) is the $\epsilon$-Nash equilibrium of the mean field game. Therefore,  the $\epsilon$-Nash equilibrium of the mean field game can be derived by solving (\ref{eq:socialwelfaremax}). To obtain a decentralized algorithm, we consider the dual form of (\ref{eq:socialwelfaremax}). The dual problem can be written as (\ref{dualprolem}), which can be efficiently solved by the following primal-dual algorithm:
 \begin{numcases}{}
   u_i\leftarrow \argmin_{u_i\in \bar{\mathcal{U}}_i} \left(V_i(x_i,u_i)+\lambda \cdot x_i \right)      \label{primaldual1}
   \\
   z\leftarrow \argmin_{z\in \mathcal{X}} \phi(z)-\lambda \cdot z      \label{primaldual2}
   \\
   \lambda\leftarrow \lambda+\beta N (\dfrac{1}{N}\sum_{i=1}^N \mathbb{E} f_i(u_i,\pi_i)-z)  \label{primaldual3} 
\end{numcases}
In the proposed algorithm, we introduce a virtual social planner to represent the cost minimization problem (\ref{primaldual2}). To implement the algorithm, an initial guess for the Lagrangian multiplier $\lambda$ is broadcast to all the agents and the social planner, who subsequently solve the primal problem. The primal problem can be decomposed over all the agents. Therefore, each agent can independently solve the stochastic programming problem (\ref{primaldual1}), while the virtual agent solves the deterministic optimization problem (\ref{primaldual2}) for a given $\lambda$. The solution of the primal problems are collected and used to update the dual $\lambda$ according to (\ref{primaldual3}). The updated dual variable is then broadcast to the agents again and this procedure is iterated until it converges to the socially optimal solution.
It can be verified that the proposed algorithm includes many existing ways to compute the mean field equilibrium special cases. For instance, the algorithms proposed in  \cite{ma2013decentralized} and \cite{rammatico2016decentralized} are equivalent to the primal-dual algorithm with a scaled stepsize in (\ref{primaldual3}), i.e., with a different value of $\beta$.  Using this insight, we can propose improved algorithms to compute the mean field equilibrium with better numerical properties, such as the alternating direction method of multipliers (ADMM). By adding proximal regularization in the Lagrangian, the ADMM method converges under more general conditions \cite{boyd2011distributed} , and converges faster than the proposed dual decomposition algorithm with appropriately selected step sizes.

\section{Case Studies}
In the case study, we consider the problem of coordinating the charging of a population of electric vehicles (EV) \cite{rammatico2016decentralized}. Each EV is modeled as a linear dynamic system, and the objective is to acquire a charge amount within a finite horizon while minimizing the charging cost. The charging cost of each EV is coupled through the electricity price, which is an affine function of the average of charging energy. This leads to the following game problem \cite{rammatico2016decentralized}:
\begin{align}
\label{PEVchargingame}
&\min_{u_i}  \eta||u_i-z||^2+2\gamma(z+c)^T u_i  \\
&\text{ subject to: } \nonumber\\
&\begin{cases}
  x_i(t+1)=x_i(t)+u_i(t) \nonumber\\
  0\leq x_i(t) \leq \bar{x}_i, 0\leq u_i(t) \leq \bar{u}_i, \sum_{t=1}^Tu_i(t)=\gamma_i, \nonumber
\end{cases}
\end{align}
where $2\gamma(z+c)^T$ denotes the electricity price, and the first term $\eta||u_i-z||^2$ penalizes the deviation from average control inputs. Note that the first term is mainly added for numerical stability \cite{rammatico2016decentralized}, we can let $\eta << \gamma$. In this example, $u_i(t)$ denotes the charging energy during the $t$th control period,  $x_i(t)\in \mathbb{R}$ is the state of charge (scaled by the capacity) of the EV battery, $\bar{x}_i$ is the battery capacity, and the electricity price is $2\gamma(z+c)$. The case study section focuses on computing the mean field equilibrium of (\ref{PEVchargingame}). We note that (\ref{PEVchargingame}) is slightly different from (\ref{individualoptimization}) in the sense that the coupling term depends on the average of control instead of the state. However, based on Remark \ref{remarkcontrol}, there is no essential difference between these two formulations, and our result applies universally. 

In \cite{rammatico2016decentralized}, the mean field equilibrium of (\ref{PEVchargingame}) is computed by a scaled version of the  algorithm (\ref{primaldual1})-(\ref{primaldual3}). If we let $\beta_k$ to denote the stepsize of (\ref{primaldual3}) during the $k$th iteration, then it is shown that the proposed algorithm converges to the mean field equilibrium of (\ref{PEVchargingame}) if $\lim_{k\rightarrow \infty}\beta_k=0$ and $\lim_{k\rightarrow \infty}\sum_{m=1}^k \beta_k=\infty$. Under this choice of stepsize, the algorithm is referred to as the Mann iteration, and we will use it as the benchmark. 

Theorem \ref{connectiontheorem1} indicates that the mean field game (\ref{PEVchargingame}) has the same solution as a convex social welfare optimization problem if $F(z)=\dfrac{1}{N} \phi'(z)$. Since $F(z)=2(\gamma-\eta)z$, then we have $\phi(z)=N(\gamma-\eta)z^Tz$. Therefore, we can compute the mean field equilibrium of (\ref{PEVchargingame}) by solving the following convex social welfare optimization problem:
\begin{align}
\label{simulation_social}
&\min_{(u_1,\ldots,u_N)} \sum_{i=1}^N \left(\eta||u_i||^2+2\gamma c^Tu_i\right)+N(\gamma-\eta)z^Tz  \\
&\text{ subject to: } \nonumber\\
&\begin{cases}
z=\dfrac{1}{N}\sum_{i=1}^N u_i \\
x_i^{t+1}=x_i(t)+u_i(t), \quad \forall i\in \mathcal{I},\quad \forall k\in \mathcal{K} \\
0\leq x_i(t) \leq \bar{x}_i, 0\leq u_i(t) \leq \bar{u}_i, \sum_{t=1}^Tu_i(t)=\gamma_i, \nonumber.
\end{cases}
\end{align}
where $\mathcal{I}=\{1,\ldots,N\}$ and $\mathcal{K}=\{1,\ldots,K\}$. To solve this problem, we propose to use alternating direction method of multipliers (ADMM). For algorithm details, please see \cite{boyd2011distributed}.

Now we compare the performance of the ADMM algorithm with the benchmark algorithm. In the simulation , we generate 100 sets of EV parameters over 36 control periods, and each period spans 5 minutes. The heterogeneous parameters, including the capacity of the EV batteries, the maximum charging rate of the batteries, and the other parameters in the objective function are all generated based on uniform distributions. We run both Mann iterations and ADMM for 200 iterations, and the simulation results are shown in Fig. \ref{fig1}-\ref{fig3}. In Fig. \ref{fig1}, we randomly select an EV and show its state of charge trajectory under the ADMM solution. In Fig. \ref{fig2} and Fig. \ref{fig3}, in order to compare the performance between ADMM and the Mann iteration, we show $||u_i||$ for a randomly selected EV and the average control input $||z||$ over each iteration. Based on the simulation results, the ADMM algorithm convergences to the optimal solution after about 50 iterations, while the Mann iteration converges after 100 iterations. Therefore, it is clear that ADMM converges faster than the benchmark algorithm. We point out that the algorithm converges only if both $z$ and $u_i$ converge. In Fig. \ref{fig3}, although the Mann iteration quickly approaches the optimal solution after a few iterations, it oscillates around the optimal solution and does not converge until after 100 iterations. This can be verified with Fig. \ref{fig2}. 


\section{conclusion}
This paper studies the connections between a class of mean field games and the social welfare optimization problem. We derived  the mean field equations for the mean field game in functional spaces, and showed that the game equilibrium coincides with the solution to a convex social welfare optimization problem. Based on the relation between the mean field game and the social welfare optimization problem, efficient algorithms can be developed to compute the mean field equilibrium by solving the corresponding  convex social welfare optimization problem. Numerical simulations are presented to validate the proposed approach. Future work includes extending the proposed approach to the case of infinitely many agents and more general formulations where the mean field term depends on the probability distribution of the population state.

\bibliographystyle{unsrt}
\bibliography{TransactiveControl}
\end{document}